\documentclass[12pt]{article}
\usepackage{amssymb, amsmath, amsthm}
\usepackage[all]{xy}

\theoremstyle{plain}
\newtheorem{theorem}{Theorem}
\newtheorem{lemma}[theorem]{Lemma}
\newtheorem{corollary}[theorem]{Corollary}

\newcommand{\Cf}{\textit{cf.}\;}

\newcommand{\gr}{\operatorname{gr}}
\newcommand{\grm}{\gr^m}

\renewcommand{\Im}{\operatorname{Im}}
\newcommand{\isomto}{\stackrel{\simeq}{\longrightarrow}}

\newcommand{\Ker}{\operatorname{Ker}}

\newcommand{\m}{\mathfrak{m}}

\renewcommand{\O}{\mathcal{O}}
\newcommand{\onto}[1]{\stackrel{#1}{\to}}
\newcommand{\OK}{\O_K}

\newcommand{\Z}{\mathbb{Z}}

\newcommand{\wt}[1]{\widetilde{#1}}

\newcommand{\Coker}{\operatorname{Coker}}
\newcommand{\grmKq}{\grm \Kq}

\newcommand{\grmkqn}{\grm\kqn}

\newcommand{\Kq}{K_q}
\newcommand{\KqM}{K_q^M}
\newcommand{\kqn}{k_{q,n}}

\newcommand{\KqMK}{\KqM(K)}

\newcommand{\UmKq}{U^m \Kq}
\newcommand{\UmppKq}{U^{m+1}\Kq}
\newcommand{\UmKqMK}{U^m \KqMK}

\newcommand{\Umkqn}{U^m\kqn}
\newcommand{\Umppkqn}{U^{m+1}\kqn}

\begin{document}
\title{Milnor $K$-groups modulo $p^n$ 
of a complete discrete valuation field}

\author{Toshiro Hiranouchi}

\maketitle
\begin{abstract}
For a mixed characteristic complete discrete valuation field 
$K$ which contains a $p^n$-th root of unity, 
we determine the graded quotients  
of the filtration on the Milnor $K$-groups $\KqMK$ modulo $p^n$ 
in terms of differential forms of the residue field of $K$.
\end{abstract}

In higher dimensional local class field theory of Kato and Parshin, 
the Galois group of an abelian extension field on 
a $q$-dimensional local field $K$ is described by 
the Milnor $K$-group $\KqMK$ for $q\ge 1$. 
The information on the ramification 
is related to the natural filtration $\UmKq := \UmKqMK$ 
which is by definition the subgroup generated by 
$\{1 + \m_K^m, K^{\times}, \ldots ,K^{\times}\}$, 
where $\m_K$ is the maximal ideal of the ring of integers $\OK$. 
So it is important to know the structure of the graded quotients 
$\grmKq := \UmKq/\UmppKq$. 
In this short note, 
we study the filtration on $\kqn := \KqMK/p^n$ 
the Milnor $K$-group modulo $p^n$  
induced by $\UmKq$. 
More precisely, 
for a mixed characteristic complete discrete valuation field 
$K$ 
Define the filtration 
$\Umkqn$ on $\kqn$, 
by the image of the filtration $\UmKq$ on $\kqn$. 
Our objective is to determine 
the structure of its graded quotient 
$\grmkqn := \Umkqn/\Umppkqn$ 
in terms of differential forms of the residue field 
under the assumption that 
$K$ contains a $p^n$-th root of unity $\zeta_{p^n}$ (Thm.\ 2). 

It should be mentioned that 
J.\ Nakamura described $\grmkqn$ 
after determining $\grmKq$ for all $m$ 
when $K$ is absolutely tamely ramified (\cite{Nak00b}, Cor.\ 1.2).
Although it is easy 
in the case of $q=1$, 
the structure of $\grmKq$ 
is still unknown in general.  
In particular, 
when $K$ has mixed characteristic 
and (absolutely) wildly ramification, 
it is known only some special cases 
(\cite{Kur04}, see also \cite{Nak00}). 
As mentioned in \cite{BK86}, Remark 6.8, 
such structure is closely related 
to the number of roots of unity 
of $p$-primary orders in $K$. 
In fact, 
Kurihara treated a wildly ramified 
field with $\zeta_p \not\in K$ in \cite{Kur04}. 
However, under the assumption $\zeta_{p^n} \in K$, 
the structure of $\grmkqn$ can be described 
by the one of $\grmKq$ only for lower $m$ 
which is known by Bloch-Kato \cite{BK86}. 
 
Let $K$ be a complete discrete valuation field of characteristic $0$, 
and $k$ its residue field of characteristic $p>0$. 
Let $e =v_K(p)$ be the absolute ramification index of $K$ and 
$e_0 := e/(p-1)$. 
For $m\ge 1$, 
let  $\UmKq$ be the subgroup of $\KqMK$ 
defined as above.
Put $U^0K_q = \KqMK$ and $\grmKq := \UmKq/\UmppKq$.
Let $\Omega_k^1 := \Omega^1_{k/\Z}$ be 
the module of absolute K\"ahler differentials  
and $\Omega_k^q$ the $q$-th exterior power of $\Omega^1_k$ over the residue field $k$. 
Define the subgroups $B_i^{q}$ and $Z_i^q$ for $i\ge 0$ of 
$\Omega_k^q$ such that 
$$
0 = B_0^q \subset B_1^q \subset \cdots \subset Z_1^q \subset Z_0^q = \Omega_k^q
$$
by the relations 
$B_1^q := \Im(d:\Omega_k^{q-1} \to \Omega_k^q)$, 
$Z_1^q := \Ker(d:\Omega_k^{q} \to \Omega_k^{q+1})$, 
$C^{-1}:B_i^{q} \isomto B_{i+1}^q/B_1^q$, and  
$C^{-1}:Z_i^q \isomto Z_{i+1}^q/B_1^q$, 
where $C^{-1}:\Omega_k^q \isomto Z_1^q/B_1^q$ 
is the inverse Cartier operator defined by 
\begin{equation}
\label{eq:iCartier}
  x\frac{dy_1}{y_1} \wedge \cdots \wedge \frac{d y_{q}}{{y_{q}}} 
    \mapsto x^p\frac{dy_1}{y_1} \wedge \cdots \wedge \frac{d y_{q}}{{y_{q}}}.
\end{equation}

We fix a prime element $\pi$ of $K$. 
For any $m$, we have a surjective homomorphism 
$\rho_m: \Omega_k^{q-1} \oplus \Omega_k^{q-2} \to \grmKq$ 
defined by 
\begin{align*}
  \left(x\frac{dy_1}{y_1} \wedge \cdots \wedge \frac{d y_{q-1}}{{y_{q-1}}}, 0\right) & \mapsto \{ 1 + \pi^m \wt{x}, \wt{y}_1, \ldots \wt{y}_{q-1}\},\\
  \left(0, x\frac{dy_1}{y_1} \wedge \cdots \wedge \frac{d y_{q-2}}{{y_{q-2}}}\right) & \mapsto \{ 1 + \pi^m \wt{x}, \wt{y}_1, \ldots \wt{y}_{q-2}, \pi\},
\end{align*}
where $\wt{x}$ and $\wt{y}_i$ are liftings of $x$ and $y_i$. 
For any $m \le e+ e_0$, 
the kernel of $\rho_m$ is written in terms of differential forms of $k$.  
Hence we obtain the structure of the graded quotient $\grmKq$ 
(\cite{BK86}, see also \cite{Nak00}) 
and also the one of $\grmkqn$ (\cite{BK86}, Rem.\ 4.8). 
Recall that the filtration 
$\Umkqn$ on $\kqn= \KqMK/p^n$ 
is defined by the image 
of the filtration $\UmKq$ on $\kqn$ 
and 
$\grmkqn := \Umkqn/\Umppkqn$. 
From the following lemma, 
we can investigate $\grmkqn$ for $m>e+e_0$ 
by its structure for $m \le e+e_0$. 

\begin{lemma}
\label{lem:m>pe0}
For $n>1$ and $m>e+e_0$, 
the multiplication by $p$ induces a surjective homomorphism
$p:U^{m-e}k_{q,n-1} \to U^m\kqn$.
If we further assume $\zeta_{p^n} \in K$, 
then the map $p$ is bijective. 
\end{lemma}
\begin{proof}
  The surjectivity of $p:U^{m-e}k_{q,n-1} \to U^m\kqn$ 
  follows from the 
  surjectivity of $p: U^{m-e}k_{1,n-1} \to U^mk_{1,n}$. 
  To show the injectivity, 
  for $x \in U^{m-e}K_q$,  
  assume that $px = p^nx'$ is in $p^n \KqMK \cap \UmKq$ 
  for some $x'\in \KqMK$. 
  Thus $x - p^{n-1}x'$ 
  is in the kernel of the multiplication by $p$ on $\KqMK$. 
  It is known that 
  its kernel is $= \{\zeta_p\}K_{q-1}^M(K)$, where $\zeta_p$ 
  is a primitive $p$-th root of unity.  
  This fact 
  is a byproduct of the Milnor-Bloch-Kato conjecture 
  (due to Suslin, \Cf \cite{Izh00}, Sect.\ 2.4), 
  now is a theorem of Voevodsky, Rost, and Weibel 
  (\cite{Wei09}).  
  Hence, for any $y\in K_{q-1}^M(K)$, we have 
  $\{\zeta_{p},y\} = p^{n-1}\{\zeta_{p^n},y\}$ and thus 
    $x \in p^{n-1}\KqMK$.  
\end{proof}

We determine $\grmkqn$ for any $m$ and $n$ when $\zeta_{p^n}$ is in $K$. 
It is known also $U^mk_{q,1} = 0$  for $m > e+e_0$ 
(\cite{BK86}, Lem.\ 5.1 (i)). 
So we may assume $m> e + e_0$ and $n>1$.
For such $m$, we have an isomorphism
$\gr^{m-e}k_{q,n-1} \onto{p} \grmkqn$ 
from the above lemma. 
By induction on $n$, 
we obtain the following:

\begin{theorem}
\label{thm:A.main}
We assume $\zeta_{p^n} \in K$. 
Let $m$ and $n$ be positive integers 
and $s$ the integer such that $m = p^sm'$, $(m',p)  = 1$. 
Put $c_i := ie + e_0$ for $i \ge 1$ and $c_0: = 0$.

\noindent
$\mathrm{(i)}$ 
If $c_i <  m < c_{i+1}$ for some $0 \le i < n$, 
then $\grmkqn$ is isomorphic to 
$$
\begin{cases}
\Coker(\Omega_k^{q-2}\!\! \onto{\theta}\! \Omega_k^{q-1}/B_s^{q-1} \oplus \Omega_k^{q-2}/B_s^{q-2}), n-i > s,\\
\Omega_k^{q-1}/Z_{n-i}^{q-1} \oplus \Omega_k^{q-2}/Z_{n-i}^{q-2},\ n-i \le s, 
\end{cases}
$$
where $\theta$ is defined by  
$\omega \mapsto (C^{-s}d\omega, (-1)^q (m-ie)/p^s C^{-s}\omega)$.

\noindent
$\mathrm{(ii)}$ 
If $m = c_i$ for some $0 < i \le n$, 
then $\gr^{ie+e_0}\kqn$ is isomorphic to 
$$
   (\Omega_k^{q-1}/(1+aC)Z^{q-1}_{n-i}) \oplus (\Omega_k^{q-2}/(1+aC)Z_{n-i}^{q-2}),
$$
where $C$ is the Cartier operator defined by 
$$
  x^p \frac{dy_1}{y_1} \wedge \cdots \wedge \frac{d y_{q-1}}{{y_{q-1}}} 
    \mapsto x\frac{dy_1}{y_1} \wedge \cdots \wedge \frac{d y_{q-1}}{{y_{q-1}}} 
$$
and $a$ is the residue class of $p\pi^{-e}$.

\noindent
$\mathrm{(iii)}$ 
If $m>c_n$, then $\Umkqn = 0$. 
\end{theorem}

\begin{corollary}
If $k$ is separably closed 
{\rm (}we may not assume $\zeta_{p^n} \in K${\rm )}, then 
$\gr^{ie+e_0}\kqn =0$ for $i\ge 1$. 
\end{corollary}
\begin{proof}
The assertion follows from $\gr^{e+e_0}k_{q,1} = 0$ (\cite{BK86}, Lem.\ 5.1 (ii)), 
Lemma \ref{lem:m>pe0}, 
and the induction on $n$. 
\end{proof}

\medskip\noindent
{\it Acknowledgments.} 
The author thanks 
Masato Kurihara for his comments 
on the results in this note.

\providecommand{\bysame}{\leavevmode\hbox to3em{\hrulefill}\thinspace}
\providecommand{\href}[2]{#2}

\noindent
Toshiro Hiranouchi \\
Department of Mathematics, Graduate School of Science, Hiroshima University\\
1-3-1 Kagamiyama, Higashi-Hiroshima, 739-8526 Japan\\
Email address: {\tt hira@hiroshima-u.ac.jp}

\end{document}